\documentclass[11pt]{amsart}
\newtheorem{theorem}{Theorem}

\theoremstyle{remark}
\newtheorem{remark}[theorem]{\bf Remark}

\newtheorem{exam}[theorem]{\bf Example}

\def\di{\operatorname{div}}

\begin{document}

\title{Ergodic and chaotic properties of some biological models}
\author{Ryszard Rudnicki}
\address{Institute of Mathematics,
Polish Academy of Sciences, Bankowa 14, 40-007 Katowice, Poland.}
\email{rudnicki@us.edu.pl}


\thanks{This research was partially supported by
the  National Science Centre (Poland)
Grant No. 2017/27/B/ST1/00100}

\begin{abstract}
In this note we present two types of biological models which have interesting ergodic and chaotic properties.
The first type are one-dimensional transformations, like a logistic map, which are used 
to describe the change in population size in successive generations. We study ergodic properties of such transformations
using Frobenius--Perron operators. The second type are some structured populations models, for example a space-structured model, 
or a model of maturity-distribution of precursors of blood cells.
These models are described by partial differential equations, which generate semiflows on the space of functions.  
We construct strong mixing invariant measures for these semiflows using stochastic precesses.
From properties of invariant measures we deduce some chaotic properties of semiflows such as the existence of dense trajectories and
strong instability of all trajectories. 
\end{abstract}

\maketitle

\section{Introduction}
\label{s:intro} 
Many biological processes behave in a chaotic manner. Classic examples are  
the discrete time population  growth model $x_{n+1}=4x_n(1-x_n)$;  
the Mackey--Glass model~\cite{Mackey-Glass}, in which the number of red blood cells changes according to a delay differential equation;
and a maturity structured model~\cite{Lasota-chaos}, in which a partial differential equation describes the evolution of the distribution of cell maturity.  
One method of studying chaos is to use the ergodic theory.
In general, the stronger the ergodic properties of a dynamical system, the more chaotic its trajectories are (see Section~\ref{s:inf-dim}). 

The aim of this note is to present some results concerning ergodic and chaotic properties of two types of biological models.
The first type are discrete time population growth models. 
Such models are described by unimodal transformations 
and a very useful tool for studying their ergodic properties are the Frobenius--Perron operators~\cite{LiM}.
In the next section, we briefly recall the definitions of the basic ergodic properties and their relation to Frobenius--Perron operators. 
In Section~\ref{s:one-dim} we formulate and prove 
a general result concerning the existence of invariant measures for a large class of unimodal transformations (see Theorem~\ref{w:f-umod}).
The invariant measure is absolutely continuous with respect to the Lebesgue measure, its density is a positive and continuous function
and the transformation is exact. We apply this theorem to the transformation $S(x)=cx(1-x^2)$, to   
the Beverton--Holt transformation $S(x)=-ax+\frac{bx}{1+x}$, and to the Ricker transformation $S(x)=-ax+axe^{\lambda(K-x)}$.

Section~\ref{s:inf-dim} is devoted to study ergodic and chaotic properties of some models of structured populations.
Such models are usually described by partial differential equations, and thus the dynamical systems generated by these equations are defined
 on infinite-dimensional spaces.
When looking for an invariant measure for an infinite dimensional dynamical system, we do not use Frobenius--Perron operators, because it is difficult to indicate the measure with respect to which the invariant measure would be absolutely continuous (would have a density).    
If we consider partial differential equations in which the spatial variable $x$ is a real number, it has turned out convenient to construct an invariant measure by means of stationary stochastic processes and isomorphism of the initial dynamic system with the shift $(T^t \varphi)(s)=\varphi(s+t)$ on a properly chosen space.
A major challenge is to construct an invariant measure if $x$ is  multidimensional. Then
instead of a stationary process, we use a random field~\cite{JDE-2023}. We show how to prove that such dynamical systems have exact invariant measures supported on the whole space. We also present conclusions on the chaotic behavior of such systems.
We apply the obtained results to various biological models:
maturity-distribution of precursors of blood cells; a size structured cellular population; and the space dispersal of a population.

\section{Measure-preserving dynamical systems}
\label{s:ergodic}
Let $(X,\Sigma,\mu)$ be a probability space, i.e. $\mu$ is a probability measure defined on a $\sigma$-algebra $\Sigma$ of subsets of $X$.
Let $\mathcal T=\mathbb N=\{0,1,\dots\}$ or $\mathcal T=[0,\infty)$.
\textit{One-parameter semigroup}\index{one-parameter semigroup}
 is a family $\{\pi_t\}_{t\in \mathcal T}$ 
of transformations of $X$
such that $\pi_s\circ \pi_t=\pi_{s+t}$ for $s,t\in \mathcal T$ and $\pi_0=\rm{Id}$. 
If additionally the function $(t,x)\mapsto \pi_t x$ is measurable as a function from  the Cartesian product $\mathcal T\times X$ to $X$
and the measure $\mu$ is
\textit{invariant}\index{invariant!measure} with respect  to each transformation $\pi_t$, $t\in \mathcal T$, i.e.
 $\mu(\pi_t^{-1}(A))=\mu(A)$ for any set $A\in\Sigma$, then the quadruple $(X,\Sigma,\mu,\pi_t)$ is called
 a \textit{measure-preserving dynamical system}.\index{measure-preserving dynamical system}
The triple $(X,\Sigma,\mu)$ is called a 
\textit{phase space}\index{phase space}
and the set
 $\{\pi_tx\colon t\in \mathcal T\}$ 
is called 
a \textit{trajectory of a point}\index{trajectory} $x\in X$.
If $\mathcal T=\mathbb N$, $S\colon X\to X$ is a measurable transformation,
the measure $\mu$ is invariant with respect to $S$, and $\pi_t=S^t$ 
is the $t$-th iterate of $S$,  then $(X,\Sigma,\mu,\pi_t)$
is a measure-preserving dynamical system.

Our goal is to study properties of measure-preserving dynamical systems:
ergodicity, mixing and exactness. 
A set $A$ is called  \textit{invariant}\index{invariant!set} with respect to a 
semigroup $\{\pi_t\}_{t\in \mathcal T}$,
if $\pi_t^{-1}(A)=A$ for any $A\in \Sigma$ and $t\in \mathcal T$. 
The family of invariant sets forms a $\sigma$-algebra  $\Sigma_{\mathrm{inv}}$.\index{Sigmaa@$\Sigma_{\mathrm{inv}}$}
If $\sigma$-algebra $\Sigma_{\mathrm{inv}}$ is 
trivial, i.e. it consists only of sets of measure $\mu$ zero or one, then   
the  measure-preserving dynamical system
 is said to be \textit{ergodic}\index{ergodic dynamical system}.   

A stronger property than ergodicity is mixing. 
A measure-preserving dynamical system  $(X,\Sigma,\mu,\pi_t)$  is called \textit{mixing}\index{mixing} if 
\begin{equation}
\label{d:mixing}
\lim_{t\to\infty} \mu(A\cap \pi_t^{-1}(B))=\mu(A)\mu(B)\quad 
\text{for  all  $A, B \in \Sigma$}.
\end{equation}
Identifying the measure $\mu$ with the probability P,  one can formulate the condition
\eqref{d:mixing} as follows:
\[   
\lim_{t\to\infty} {\rm P}(\pi_t(x)\in A|x\in B)={\rm P}(A)\quad
\text{for  all $A,B \in \Sigma$ and ${\rm P}(B)>0$},
\]
which means that 
the trajectories of almost all points enter 
a set $A$ with  asymptotic probability ${\rm P}(A)$.

A stronger property than mixing is exactness.
A system $(X,\Sigma,\mu,\pi_t)$  with  double measurable transformations $\pi_t$, i.e. $\pi_t(A)\in \Sigma$ and 
$\pi_t^{-1}(A)\in \Sigma$ for all $A\in \Sigma$ and $t\in \mathcal T$,  and with  an invariant probability measure 
$\mu$ is called 
\textit{exact}\index{exact}
if for every set 
$A \in \Sigma$ with $\mu(A)>0$ 
we have $\lim_{t\to\infty} \mu(\pi_t(A))=1$.
Exactness is equivalent to the following condition:
 $\sigma$-algebra 
$\bigcap_{t \ge 0}  \pi_t^{-1}(\Sigma)$ 
contains only sets of measure $\mu$ zero or one.
Here $\pi_t^{-1}(\Sigma)=\{\pi_t^{-1}(A)\colon A\in \Sigma\}$.

The ergodic properties of the dynamical system can be studied  using Frobenius--Perron operators~ \cite{LiM}.
Let $(X,\Sigma,m)$ be a $\sigma$-finite measure space.
A measurable map $\varphi\colon X\to X$  
is called \textit{non-singular}\index{non-singular map}
if it satisfies the following condition 
\begin{equation}
\label{trana2}
m(A)=0 \Longrightarrow  m(\varphi^{-1}(A))=0\textrm{ for $A\in \Sigma$.}
\end{equation}
Let $L^1=L^1(X,\Sigma,m)$ and let $\varphi$ be a measurable nonsingular transformation of $X$.
An operator  ${P\!}_{\varphi}\colon  L^1\to L^1$
which satisfies the following condition
\begin{equation}
\label{def-FP}
\int_A {P\!}_{\varphi} f(x)\,m(dx)=
\int_{\varphi^{-1}(A)} f(x)\,m(dx)
\textrm{\,  for $A\in \Sigma$ and $f\in L^1$}
\end{equation}
is called the
{\it Frobenius--Perron operator} for the transformation~$\varphi$. 
The operator ${P\!}_{\varphi}$ is linear, \textit{positive}\index{positive!operator} 
(if $f\ge 0$ then ${P\!}_{\varphi}f \ge 0$) and preserves the integral    
($\int_X{P\!}_{\varphi}f\,dm=\int_Xf\,dm $).
The adjoint of the Frobenius--Perron operator
${P}_{\!\varphi}^*\colon L^{\infty}\to L^{\infty}$ is 
given by ${P}_{\!\varphi}^*g(x)= g(\varphi(x))$.

Let    $\varphi\colon X\to X$ be a nonsingular transformation of a 
$\sigma$-finite measure space $(X,\Sigma,m)$.
Denote by  $D$ the set of all densities with respect to $m$,
i.e. functions $f\in L^1(X,\Sigma,m)$ such that $f\ge 0$ and $\|f\|=1$.
Let $f^*\in D$. Then the measure 
$\mu(A)=\int_Af^*\,dm$ for $A\in \Sigma$,
 is invariant with respect to $\varphi$  
if and only if ${P\!}_{\varphi}f^*=f^*$.
If the map
$\pi\colon \mathcal T\times X\to X$ is measurable, 
 $\{\pi_t\}_{t\in \mathcal T}$ is
a one-parameter semigroup of nonsingular transformations of $(X,\Sigma,m)$, and
$P^t$ denotes the Frobenius--Perron operator corresponding to $\pi_t$,
 then the quadruple $(X,\Sigma,\mu,\pi_t)$ 
is a measure-preserving dynamical system
if and only if $P^t\!f^*=f^*$ for all $t\in \mathcal T$.
We collect the relations between ergodic properties of dynamical systems 
$(X,\Sigma,\mu,\pi_t)$ and the behavior of the 
Frobenius--Perron operators $P^t$ in Table~\ref{tab1}.
\begin{table}[h]
\begin{tabular}{|c|c|}
\hline
$\mu$ & $f^*$
\\
\hline
invariant&$P^t\!f^*=f^*$ for all $t\in \mathcal T$\\
ergodic&$f^*$ is a unique fixed point  in $D$ of all $P^t$ \\
mixing& w-$\lim_{t\to\infty} P^t\!f=f^*$ for every $f\in D$\\
exact & $\lim_{t\to\infty} P^t\!f=f^*$ for every $f\in D$\\
\hline
\end{tabular}
\caption{}
\label{tab1}
\end{table} 

We recall that the \textit{weak limit}\index{weak limit}  w-$\lim_{t\to\infty} P^t\!f$ is a function $h\in L^1$  such that for 
every $g\in L^{\infty}$ we have 
\[
\lim_{t\to\infty}\int_X  P^t\!f(x)g(x)\,m(dx)=\int_X  h(x)g(x)\,m(dx).
\]

Now we show how to find the Frobenius--Perron operator for a piecewise smooth 
transformation $\varphi$ of some interval $\Delta$.
We assume that there exists 
at most countable family of pairwise disjoint open intervals
$\Delta_i$, $i\in I$, contained in $\Delta$ and having the following properties:
\begin{enumerate}
\item[a)] the sets  $\Delta_0=\Delta\setminus \bigcup_{i\in I} \Delta_i$ and
 $\varphi(\Delta_0)$ have zero Lebesgue measure,      
 \item[b)] maps $\varphi_i =\varphi\Big |_{\Delta_i}$ are 
 $C^1$-maps 
from $\Delta_i$ onto $\varphi(\Delta_i)$ and $\varphi'_i(x)\ne 0$
for  $x\in \Delta_i$.      
\end{enumerate}

\noindent Then the Frobenius--Perron operator ${P\!}_{\varphi}$ exists and is given by the formula
\begin{equation}
\label{F-P-operator-gladki}
{P\!}_{\varphi}f(x)=\sum_{i\in I_x} f(\psi_i(x))|\psi_i'(x)|,
\end{equation} 
where $\psi_i=\varphi_i^{-1}$ and $I_x=\{i\colon x\in \varphi(\Delta_i)\}$ (see \cite{Rudnicki-LN,RT-K-k}). 

\begin{exam}
\label{tent-map}
Let us consider the \textit{tent map}\index{tent map} 
$\varphi\colon [0,1]\to [0,1]$ given by
\begin{equation}
\label{tent}
\varphi(x)=
\begin{cases}
2x &\textrm{for $x\in [0,1/2]$},\\
2-2x&\textrm{for $x\in (1/2,1]$.}
\end{cases}
\end{equation} 
The Frobenius--Perron operator ${P\!}_{\varphi}$ is of the form
\begin{equation}
\label{FP-diad}
{P\!}_\varphi f(x)= \tfrac 12 f(\tfrac 12x)+\tfrac 12f(1-\tfrac 12x).
\end{equation} 
It is easy to see that the density $f^*=1_{[0,1]}$ satisfies ${P\!}_\varphi f^*=f^*$, and therefore the Lebesgue measure on $[0,1]$ is invariant with respect $\varphi$.
We check that $\lim_{t\to\infty}{P}^t_{\!\varphi}f=f^*$ for any density $f$.
It is sufficient to check  this condition 
for densities which are Lipschitz continuous. 
Let $L$ be  the Lipschitz constant for $f$. Then
\[
|{P\!}_{\varphi} f(x)-{P\!}_{\varphi}f(y)|\le \tfrac 12
 |f(\tfrac{x}2) -f(\tfrac{y}2)|   +\tfrac 12  | f(1-\tfrac 12x)+ f(1-\tfrac 12y)|\le  \tfrac L2 |x-y|.   
\]
Thus   $L/2$ is the  Lipschitz constant for  ${P\!}_{\varphi}f$ and 
 by induction we conclude that $L/2^t$ is  the  Lipschitz constant for  
${P}^t_{\!\varphi}f$.
Hence, the sequence $({P}^t_{\!\varphi}f)$ converges uniformly to a constant function.
Since ${P}^t_{\!\varphi}f$  are densities, $({P}^t_{\!\varphi}f)$ converges to $f^*$ uniformly, which implies 
the convergence in $L^1$.
The condition $\lim_{t\to\infty}{P}^t_{\!\varphi}f=f^*$
implies the exactness of the transformation~$\varphi$.
\end{exam}

We now introduce the notion of an isomorphic dynamical system, which we use extensively  in the next section. 

Let  $(X,\Sigma,\mu,\pi_t)$  be a measure-preserving dynamical system,
$(\tilde X,\tilde \Sigma)$ be a measurable space and 
$\tilde \pi \colon \mathcal T\times \tilde X\to \tilde X$ be a measurable one-parameter semigroup. Assume that there exists an invertible double measurable function $\alpha\colon X\to \tilde X$ such that $\alpha\circ \pi_t=\tilde \pi_t\circ\alpha$ 
for $t\in \mathcal T$. Let 
$\tilde \mu(A)=\mu(\alpha^{-1}(A))$ for $A\in \tilde \Sigma$.   
Then $(\tilde X,\tilde \Sigma,\tilde \mu,\tilde \pi_t)$ is 
a measure-preserving dynamical system 
\textit{isomorphic}\index{isomorphic dynamical system} to 
$(X,\Sigma,\mu,\pi_t)$. If two dynamical systems are isomorphic, then they have the same ergodic properties.

\begin{exam}
\label{log-map}
Now we check that the \textit{logistic map}\index{logistic map}  $\psi(x)=4x(1-x)$ on $[0,1]$ is exact. 
Let $\varphi$ be the tent map and $\alpha(x)=\tfrac12-\tfrac 12\cos(\pi x)$. Then $\psi\circ \alpha=\alpha\circ \varphi$. 
Let $\mu(A)=|\alpha^{-1}(A)|$
for $A\in\mathcal B([0,1])$, where $|\cdot|$ is the Lebesgue measure
on the $\sigma$-algebra $\mathcal B([0,1])$ of Borel subsets of $[0,1]$. This implies that the map $\psi$
is an exact transformation of the space $([0,1],\mathcal B([0,1]),\mu)$. Since $\alpha^{-1}(x)=\pi^{-1}\arccos(1-2x)$ we find
that $d\mu=g^*(x)\,dx$, where
\[
g^*(x)=\frac{d}{dx}\alpha^{-1}(x) =\Big[\pi\sqrt{x(1-x)}\,\Big]^{-1}.
\]
\end{exam}

More advanced examples of applications of Frobenius--Perron operators to study ergodic properties of dynamical systems 
can be found in~\cite{BG,LiM,LY}. 

\section{Population generation models}
\label{s:one-dim}
Consider a population model in which the number of individuals in successive generations changes according to a recursive formula
\begin{equation}
\label{m:pok-rek}
x_{n+1}=S(x_n),
\end{equation}
$S\colon [0,K]\to [0,K]$ is a continuous function of the form  $S(x)=bxf(x)$, where $b$  is the 
\textit{per capita birth rate},\index{per capita birth rate}
 $K$ is the 
\textit{capacity of the environment},\index{capacity of the environment}
 in this case we identify it with the maximum population size, and $f(x)$ is a 
\textit{competitive function}\index{competitive function} describing the probability of survival if $x$ is the population size.
This is a \textit{model with disjoint generations}.  We assume that for some $\tilde x\in (0,K]$ we have $S(\tilde x)=K$, because otherwise the maximum population size would be smaller than $K$.

We will consider specific transformations $S$ appearing in known models, in which $S(0)=S(K)=0$ and $S$ is a strictly increasing function in the interval $[0,\tilde x]$ and strictly decreasing in the interval $[\tilde x,K]$.
Such transformations are called 
\textit{unimodal}.\index{unimodal transformation}

Example~\ref{log-map}
suggests that exactness of a unimodal transformation $S\colon [0,K]\to [0,K]$ can be studied by showing the
isomorphism of  $S$ with the tent map $\varphi$, and thus finding such a function $\Phi$ strictly increasing from the interval $[0,1]$ to the interval $[0,K]$ (preferably of class $C^1$)
satisfying the condition $S\circ \Phi=\Phi\circ \varphi$. Unfortunately, this method is not easy, so we need to modify it.
To do this, we present a theorem on the exactness of a certain class of piecewise expanding  maps of an interval $[0,L]$.

\begin{theorem}
\label{tw-o-doklad-kaw-rozsz}
Let $S\colon [0,L]\to [0,L]$ be a map satisfying the following conditions:
\begin{enumerate}
\item[\rm(i)] there exists a partition $0 = a_0 < a_1 < \dots < a_r = L$ of the interval  $[0,L ]$ such that for each   $i = 1, \dots, r$
the restriction of $S$ to the interval  $(a_{i-1},a_i)$ is a  $C^2$-function,
\item[\rm(ii)]  $\,S((a_{i-1},a_i)) =(0,L)$ for $i=1,\dots,r$, 
\item[\rm(iii)] $\,|S'(x)|>\lambda$ for some $\lambda>1$ and for all  $x\ne a_i$, $i=0,\dots,r$,
\item [\rm(iv)] $\,|S''(x)|/[S'(x)]^2<c$ for some  $c>0$ and for all $x\ne a_i$, $i=0,\dots,r$.
\end{enumerate}
Then the Frobenius--Perron operator $P_S$ is asymptotically stable,
i.e., there exists $f^*\in D$ such that  $\lim_{n\to\infty} P_S^nf=f^*$ for every $f\in D$.
\end{theorem}

\begin{remark}
If the map $S$ satisfies conditions (i), (iii), and 
\begin{enumerate}
 \item[{\rm(iv}${}'$)]  the map $S$ has $C^2$ extensions from open intervals $(a_{i-1},a_i)$ onto closed intervals $[a_{i-1},a_i]$,
 for $i=1,\dots,r$,
\end{enumerate}
then $S$ has  an invariant measure absolutely continuous with respect to the Lebesgue measure (see \cite{LY1973}). 
The condition (iv) follows immediately from (iv${}'$). 
Moreover if $S$ satisfies condtions  (i),(ii),(iii),(iv${}'$), then
the invariant  density  $f^*$ is a continuous function and satisfies inequalities
\begin{equation}
\label{oszac-gest-niezm} 
M^{-1}\le f^*(x)\le M\,\,\,  \text{for some $M>0$ and for each $x\in [0,L]$.}
 \end{equation}
\end{remark}

The proof of Theorem~\ref{tw-o-doklad-kaw-rozsz} is given in~\cite[Th.  6.2.2]{LiM}.
The proof of the property (\ref{oszac-gest-niezm}), as well as a number of other interesting results on one-dimensional transformations can be found in the  monograph~\cite{BG}. 

Now we give sufficient conditions for exactness of unimodal transformations. 

\begin{theorem}
\label{w:f-umod}
We assume that $S$ is a $C^3$ unimodal function 
 satisfying the condition
\begin{equation}
\label{war-barS-gladkie}
S'(x)> 0 \text{ for $x\in [0,\tilde x)$},\quad S'(x)<0 \text{ for $x\in (\tilde x,K]$},\quad  S''(\tilde x)< 0.
\end{equation}
Let 
\begin{equation}
\label{funkcja-h-il}
h_1(x)=\sqrt{\frac{x(K-x)}{S(x)}},\quad h_2(x)= \frac{|S'(x)|}{\sqrt{(K-S(x))}}.
\end{equation}
If $\,\inf h_1h_2>1$, 
then there is a probability measure $\mu$ absolutely continuous with respect to the Lebesgue measure such that   
the quadruple $([0,K],\mathcal B([0,K]),\mu,S)$ is a  measure-preserving  dynamical system. This system is exact,
and the density $g^*$ of the measure $\mu$ is a positive and continuous function on the interval $(0,K)$ and 
\begin{equation}
\label{funkcja-g*}
\lim_{x\to 0^+} g^*(x)=\lim_{x\to K^-} g^*(x)=\infty. 
\end{equation}
\end{theorem}

\begin{proof}
 Let $\Phi$ be the  function from the interval $[0,\pi]$ onto the interval $[0,K]$ given by $\Phi(x)=\frac K2(1-\cos x)$. 
Then $\Phi$ is strictly increasing and we can define the transformation $\tilde S\colon [0,\pi]\to [0,\pi]$ by  $\tilde S(x)=\Phi^{-1}\circ S\circ \Phi$.
We check that $\tilde S$ satisfies conditions (i), (ii),  (iii), and  (iv${}'$).
Let  $\Psi=\Phi^{-1}$ and $\psi=\Psi'$. 
Then
\[
\tilde S'(x)=\Psi'((S\circ\Phi)(x))S'(\Phi(x))\Phi'(x)=\frac{\Psi'(S(y))S'(y)}{\Psi'(y)}=\frac{\psi(S(y))S'(y)}{\psi(y)},
\]
where $y=\Phi(x)$. Let 
\[
h(y)=\frac{\psi(S(y))S'(y)}{\psi(y)}.
\]
Then 
\[
\tilde S''(x)=h'(\Phi(x))\Phi'(x)=\frac{h'(y)}{\psi(y)}.
\]
Since 
\begin{equation}
\label{psi-spec}
 \Psi(x)=\Phi^{-1}(x)=\arccos(1-2x/K), \quad \psi(x)=\Psi'(x)=\frac{1}{\sqrt{x(K-x)}},
 \end{equation}
we have
\begin{equation}
\label{funkcja-h}
|h(x)|=\frac{\psi(S(x))|S'(x)|}{\psi(x)}=\frac{\sqrt{x(K-x)}\,|S'(x)|}{\sqrt{S(x)(K-S(x))}}=h_1(x)h_2(x).
\end{equation}

Note that condition (iii) holds when $\inf |h|=\inf h_1h_2 >1$.
We check that $h_1$ and $h_2$ are  $C^1$ functions.
Indeed, the function $h_1$ can only have singularities at the points $x=0$ and $x=K$, and since 
$S(0)=S(K)=0$ and $S'(0)\ne 0$, $S'(K)\ne 0$,  the function $S$ can be represented as the product 
$S(x)=x(K-x)s_1(x)$, where $s_1$ is a strictly positive $C^1$ function, so 
 $h_1(x)=1/\sqrt{s_1(x)}$ is a $C^1$ function.
The square of the function $h_2$ is the quotient of the functions $S'^2$ and $K-S$ and can only have a singularity at the point $\tilde x$.
Since $S'(\tilde x)=0$,  we can represent the function $S'$ as $S'(x)=S''(\tilde x)(x-\tilde x)s_2(x)$,
where $s_2$ is a $C^1$  function. Since $\tilde x$ is the zero of the function $K-S$ and its first derivative, 
the function $K-S$ can be written 
in the form $K-S(x)=\frac12 S''(\tilde x)(x-\tilde x)^2s_3(x)$, where again $s_3$ is a $C^1$ function different from zero.
We can see that 
\[
h_2^2(x)=\frac{[S'(x)]^2}{K-S(x)}=2S''(\tilde x)\frac{s_2^2(x)}{s_3(x)}
\]
and $s_2(x)>0$ for $x\in [0,K]$, so $h_2$ is a $C^1$ function. 
Since $1/\psi(x)=\sqrt{x(K-x)}$ is a continuous function, the condition  (iv${}'$)  holds.

Since $S$ is a unimodal function, conditions (i) and (ii) follow immediately from (\ref{war-barS-gladkie}).
If $f^*$ is the invariant density of the mapping $\tilde S$, then 
\[
g^*(x)=f^*(\Psi(x))|\Psi'(x)|=\frac{1}{\sqrt{x(K-x)}}
\,f^*(\arccos(1-2x/K))
\]
is the invariant density of  $S$. Thus $g^*$ is a continuous function in the interval $(0,K)$ 
and satisfies \eqref{funkcja-g*}.
\end{proof}

\begin{exam}
\label{b:trans-trzeciego-st}
Let us consider the transformation $S(x)=cx(1-x^2)$ defined on the interval $[0,1]$.
Since $S'(x)=c(1-3x^2)$, we find that  $\tilde x=\frac{\sqrt{3}}{3}$ and from the equation $S(\tilde x)=1$ that
 $c=\frac{3\sqrt{3}}{2}$.
The polynomial $1-S(x)=1-cx(1-x^2)$
has a double zero at $\tilde x$, so its decomposition into factors is of the form:
\[
1-S(x)=c\Big(x-\frac{\sqrt{3}}{3}\Big)^2\Big(x+\frac{2\sqrt{3}}{3}\Big).
\]
Thus, we easily determine $h$ and check the inequality $\inf |h|>1$:
\[
\begin{aligned}
|h(x)|&=\frac{\sqrt{x(1-x)}\,|S'(x)|}{\sqrt{S(x)(1-S(x))}}=
\frac{c|1-\sqrt 3 x|(1+\sqrt 3 x)}{\sqrt{c^2(1+x)\Big(x-\frac{\sqrt{3}}{3}\Big)^2\Big(x+\frac{2\sqrt{3}}{3}\Big)}}\\
&= \frac{3\Big(x+\frac{\sqrt 3}{3}\Big )}{\sqrt{(x+1)\Big(x+\frac{2\sqrt{3}}{3}\Big)}}\ge 
 \frac{3\frac{\sqrt 3}{3}}{\sqrt{\frac{2\sqrt{3}}{3}}}=\frac{3}{\sqrt{2\sqrt{3}}}>1.
\end{aligned} 
\]
\end{exam}

\begin{exam}
\label{transformacja-BH}
Consider the 
\textit{Beverton--Holt transformation}\index{Beverton--Holt transformation}  
\[
S(x)=-ax+\frac{bx}{1+x}.
\]
Observe that  $S(K)=0$ if $b=(K+1)a$.
Since $S'(x)=-a+b/(1+x)^2$, the function  $S$ has the maximum at $\tilde x=\sqrt{b/a}-1=\sqrt{K+1}-1$.
Hence
\[
S(\tilde x)=-a\tilde x+\frac{a(K+1)\tilde x}{1+\tilde x}=-a\tilde x+a\tilde x\sqrt{K+1}=a(\sqrt{K+1}-1)^2,
\] 
and since $S(\tilde x)=K$, we have
\[
a=\frac{K}{(\sqrt{K+1}-1)^2},\quad b=\frac{K(K+1)}{(\sqrt{K+1}-1)^2}.
\]
Then
\[
\begin{aligned}
h_1(x)&=\sqrt{\frac{x(K-x)}{S(x)}}=\sqrt{\frac{K-x}{-a+\frac{b}{1+x}}}=\sqrt{\frac{(K-x)(1+x)}{-a(1+x)+a(K+1)}}=
\sqrt{\frac{1+x}{a}},
\\
h_2(x)&=\frac{|S'(x)|}{\sqrt{K-S(x)}}=\frac{\Big|-a+\frac{b}{(1+x)^2}\Big|}{\sqrt{K +ax-\frac{bx}{1+x}}}
=\frac{a\Big|(x+1)^2-\frac{b}{a}\Big|(x+1)^{-3/2}}{\sqrt{(ax+K)(x+1)-bx}}\\
&=\frac{a\big|(x-\tilde x)\big(x+1+\sqrt{K+1}\,\big)\big|(x+1)^{-3/2}}{\sqrt{a(x-\tilde x)^2}}\\
&=\sqrt a\,\big(x+1+\sqrt{K+1}\,\big)(x+1)^{-3/2},
\end{aligned}
\]
hence
\[
h_1(x)h_2(x)=\frac{x+1+\sqrt{K+1}}{x+1}\ge \frac{K+1+\sqrt{K+1}}{K+1}>1.
\]
\end{exam}

In some examples it is difficult to find the greatest lower bound of the function  $h_2$.  
If $\inf |S''|>0$, then we can estimate $h_2$ from below by applying Cauchy's mean value theorem:
if functions $f$ and $g$ are continuous in the interval  $[a,b]$
and differentiable in the interval $(a,b)$, then there is a point 
$c\in (a,b)$ such that
\begin{equation}
\label{tw-Cauchy-wzor}
(f(b)-f(a))g'(c)=(g(b)-g(a))f'(c).
\end{equation}
Let $f(x)=S'^2(x)$, $g(x)=K-S(x)$ and we choose $a=\tilde x$, $b=x$, if $x>\tilde x$, and $b=\tilde x$, $a=x$, if $x<\tilde x$.
In both cases   
\[
\frac{S'^2(x)}{K-S(x)}=\frac{2S''(c)S'(c)}{-S'(c)}=-2S''(c)
\]    
for some $c\in [0,K]$. Thus $ h_2\ge \min\sqrt{2|S''|}$.

\begin{exam}
\label{transformacja-Rickera}
Consider the  \textit{Ricker transformation}\index{Ricker transformation}
\[
S(x)=-ax+axe^{\lambda(K-x)}.
\]
It satisfies the condition $S(0)=S(K)=0$. 
We  change coordinates linearly to get  the maximum of $S$ at $\tilde x=1$,
and then we find the relations between $a$, $K$, and $\lambda$. Since $S(1)=K$ and $S'(1)=0$, we obtain
\[
K=1-\lambda^{-1}\ln(1-\lambda),\quad a=\frac{(1-\lambda)K}{\lambda}. 
\]  
Let us observe that 
\begin{align*}
h_1(x)&=\sqrt{\frac{x(K-x)}{S(x)}}=\sqrt{\frac{K-x}{ae^{\lambda(K-x)}-a}}\,\ge \sqrt{\frac{K}{ae^{\lambda K}-a}}\,,
\\
S''(x)&=\lambda a(\lambda x-2)e^{\lambda(K-x)}.
\end{align*}
Assume that $\lambda K<2$. Then $S''$ is a negative and increasing function. Hence
\begin{align*} 
|S''(x)|&\ge \lambda a|\lambda K-2|,
 \\ 
h_2(x)&\ge \sqrt{2\lambda a(2-\lambda K)}.
\end{align*}
From these inequalities we obtain
 \[
|h(x)|\ge \sqrt{\frac{2\lambda K(2-\lambda K)}{e^{\lambda K}-1}}.
\] 
Let $c_0>0$ be a constant such  that $2c_0(2-c_0)=e^{c_0}-1$. Then $c_0\approx 1.0928$ and 
$\min |h|>1$ for $\lambda K<c_0$.
Since $\lambda K=\lambda-\ln(1-\lambda)$,
the inequality 
\[
\lambda-\ln(1-\lambda)<c_0
\]
implies that $\min |h|>1$.
We can now choose $\lambda_0\approx 0.4658$ such that $\min h>1$ for $\lambda<\lambda_0$.
\end{exam}

\section{Invariant measures and chaos of structured populations}
\label{s:inf-dim}
Structured populations are usually described by partial differential equations. 
The dynamical systems generated by such equations
are defined on infinite dimensional spaces.
As an example we consider the following equation
\begin{equation}
\label{liniowe}
\frac{\partial u}{\partial t}+x\frac{\partial u}{\partial x}= \lambda u,\quad x\in [0,1], \quad t\ge 0, \quad \lambda >0
\end{equation}
with the initial condition $u(0,x)=v(x)$.
This equation generates a semiflow
$\{S^t\}_{t\ge 0}$ on the space
$X=\{v\in C[0,1]\colon \ v(0)=0\}$
given by the formula
$S^tv(x)=e^{\lambda t}v(e^{-t}x)$.
One method of study of ergodic properties of this semiflow 
developed by Lasota~\cite{Lasota} is based on 
the Krylov--Bogolyubov theorem  
on the existence of invariant measures for dynamical systems on compact spaces. 
Unfortunately an invariant measure constructed by this method 
does not have sufficiently strong ergodic and analytic properties.

Another method developed 
in the papers~\cite{BK,Rudnicki-pde1,Rudnicki-pde2} 
is based on the following observation. Let
$Q\colon X\to C[0,\infty)$ be the map given by 
$(Qv)(t)=S^tv(1)=e^{\lambda t}v(e^{-t})$
and let 
$\{T^t\}_{t\ge 0}$ be  the left-side shift on the space
$Y=Q(X)$ defined by $(T^t \varphi)(s)=\varphi(s+t)$.
Then $T^t\circ Q=Q\circ S^t$ for $t\ge 0$.
Let 
$\xi_t$, $t\ge 0$, be a stationary stochastic process defined on a probability space $(\Omega,\Sigma,{\rm P})$
 whose  sample paths  are in  $Y$.
 Then the measure $m(A)={\rm P}(\omega\colon \xi_.(\omega)\in A)$
 defined on the Borel $\sigma$-algebra  $\mathcal B(Y)$
 is invariant with respect to $\{S^t\}_{t\ge 0}$.
Since  dynamical systems $\{S^t\}_{t\ge 0}$
and $\{T^t\}_{t\ge 0}$ 
are isomorphic the measure $\mu(A)=m(Q(A))$ is invariant 
with respect to $\{S^t\}_{t\ge 0}$.

Let $w_t$, $t\ge 0$,  be a Wiener process  starting from $0$
with continuous sample paths and let 
$\xi_t=e^{\lambda t} w_{e^{-2\lambda t}}$
for $t\ge 0$. Then  $\xi_t$, $t\ge 0$, is a Gaussian stationary
process with continuous paths. If $m$ is the measure induced by 
the process $(\xi_t)$, then the measure $\mu$ is induced by the process $\zeta_x=w_{x^{2\lambda}}$, $x\in [0,1]$.
Since the Wiener measure is  positive on nonempty open subsets of $X$
the same property has the measure $\mu$.

Now we check that the dynamical system
$(X,\mathcal B(X),\mu,S^t)$ is exact.
Denote
by $\mathcal F_T$ the $\sigma$-algebra of events generated by
the process $w_t$ for $t\in T$.
From the definition of the dynamical system $\{S^t\}_{t\ge 0}$ it follows that 
$S^{-t}(\mathcal B(X))\subseteq  \mathcal F_{[0,e^{-2\lambda t}]}$
for $t\ge 0$. According to 
Blumenthal's zero-one law the $\sigma$-algebra
$\mathcal F_{0^+}=\bigcap_{s>0} \mathcal F_{[0,s]}$
contains only sets of measure zero, the same property has the 
$\sigma$-algebra
 $\bigcap_{t\ge 0} S^{-t}(\mathcal B(X))$, which proves 
that the dynamical system
$(X,\mathcal B(X),\mu,S^t)$ is exact.

In summary, the dynamical system $\{S^t\}_{t\ge 0}$ has an invariant measure,
which is positive on nonempty open sets (property (P)) and this system is exact. 
These properties imply strong chaotic behaviour of this dynamical system. From (P) and ergodicity  it follows that $\mu$-almost all  
trajectories are dense. From (P) and from the mixing property 
it follows a  strong version of unstability
called \textit{sensitive dependence on initial conditions}:
that 
there exists a constant $\eta>0$ such that for each 
point $v\in X$ and for each $\varepsilon>0$ there exist a point 
$\bar v$ and  $t>0$, such that  $\|v-\bar v\|<\varepsilon$  
and $\|S^tv-S^t\bar v\|>\eta$. 

Our system has also turbulent trajectories in the sense of Bass.
We recall that a trajectory $\mathcal O(v)=\{S^tv\colon t\ge 0\}$ of  $v\in X$ is 
\textit{turbulent in the sense of Bass}\index{turbulent in the sense of Bass} \cite{Bass}
if there exists an $v_0\in X$ and a function  
$\gamma\in C([0,\infty),X)$  
such that
\begin{enumerate}
\item[(I)]
$\lim\limits_{T\to\infty}
\frac1T
\int_0^T S^tv\,dt=v_0$, 
\item[(II)]
$\lim\limits_{T\to\infty}
\frac1T
\int_0^T (S^tv-v_0) 
(S^{t+\tau}v-v_0) \,dt = \gamma (\tau)$,
\item[(III)]
$\gamma (0)\neq 0$ and 
$\lim_{\tau\to\infty}
\gamma (\tau)= 0$.
\end{enumerate}
The number  $\gamma(\tau)$
describes the correlation between the trajectory 
$\mathcal O(v)$
and its $\tau$-time shift $\mathcal O(S^{\tau}v)$.
Condition 
$\lim_{\tau\to\infty}\gamma (\tau)= 0$
can be interpreted as the ``lack of memory" because 
the trajectory and its $\tau$-time shift are 
``almost independent" for large $\tau$.
The proof of this property is more advanced \cite{Rudnicki-pde2}.  
We need to check that 
the second moment of $\mu$ is finite, i.e.
$\int_{X} \| v \|^2 \mu (d v) < \infty$
and then apply an ergodic theorem for Banach space valued random variables to prove (I) and (II), and mixing property to prove (III).
 
We can consider a dynamical system $\{S^t\}_{t\ge 0}$ restricted to the set 
$X_+=\{v\in X\colon v\ge 0\}$. Then 
the measure $\mu$ is induced by the process $\zeta_x=|w_{x^{2\lambda}}|$, $x\in [0,1]$,
is invariant with respect to $\{S^t\}_{t\ge 0}$, positive on nonempty open sets
and the system $(X_+,\mathcal B(X_+),\mu,S^t)$ is exact. It means that this system has all above mentioned chaotic properties.

Now we apply this observation to study two models of hematopoietic system.
The proper functioning of the hematopoietic system largely depends on the process of multiplication and differentiation (maturation) of erythrocytes. 
The process of erythrocyte production itself is quite complicated, as 
the formation of a mature erythrocyte occurs as a result of multiple division and differentiation of erythroid stem cells. 
This process is regulated by the concentration of erythropoietin. 
It turns out that when the external regulatory mechanisms of the hematopoietic system do not work, the system can act chaotically.

Although our models are rather simple in comparison with
other models for erythroid production (e.g. \cite{LMW}, \cite{DM}),
they are based on the same continuous maturation-proliferation scheme.
In both models we  assume that 
the level of morphological development, briefly referred as the cell \textit{maturity},
is expressed by a parameter $x$ in the interval $[0,1]$.
We assume that maturity is a real number $x\in [0,1]$. 
The function $u(t,x)$ describes the  distribution 
of cells with respect to their maturity.
Assume that maturity grows according to the  equation
$x'=g(x)>0$, $g(0)=0$. We assume that at the time of division, the daughter cells
have the same maturity as the parent cell.
When the cell reaches maturity $x=1$ it leaves the bone marrow.
In the first model, we assume that 
each cell splits
with rate $c$ or dies  with rate
 $d$. 
Let $c=b-d$. 
Then the distribution $u(t,x)$ of the population 
at time $t$ relative to maturity $x$
satisfies the equation 
\begin{equation}
\label{prekursor-krwi-1-n}
\frac{\partial u} {\partial t} + \frac{\partial}{\partial x}
(g ( x)  u ) =c u.
\end{equation}
This model can be found in the work~\cite{LMW}.
We will consider here a simplified version with $g(x)=x$.
Then Eq.~\eqref{prekursor-krwi-1-n} can be written in the form
of Eq.~\eqref{liniowe} with $\lambda=c-1$. Thus if $c>1$ this model is chaotic.

Now we consider the second model~\cite{Rudnicki-chaos-den}.
We assume that when one cell reaches the maturity $x=1$ 
it leaves the bone marrow, and then 
one of cells from the bone marrow splits. 
This cell is chosen randomly according to the distribution given by the density
$p(t,x)$ of all stem cells. Let $D_0$ be the subset of densities $p$ such that 
$\int_0^{\varepsilon}p(x)\,dx>0$ for each $\varepsilon>0$.
We need to assume that the initial density $p_0(x)=p(0,x)$ belongs to $D_0$.
Otherwise, the stem cell population will die out in finite time.
Then for each $t>0$ also the density $p(t,x)$ belongs to $D_0$.

If the density $p_0$ is a differentiable function, then $p$ satisfies a nonlinear partial differential equation
\begin{equation}
\label{e-m-pr-nielinowy}
\frac{\partial p}{\partial t}+
\frac{\partial }{\partial x} (xp)=
p(t,1)p(t,x).
\end{equation}
The difference between the equations~(\ref{prekursor-krwi-1-n}) and (\ref{e-m-pr-nielinowy})
is that in the first model, the growth rate is $c$, while in the second model 
this rate is equal to $u(t,1)$. 
This form of the rate comes from the fact that $u(t,1)$ 
describes the number of cells leaving the bone marrow per unit time
and this is also how many new cells appears.
Is not difficult to check that
\begin{equation}
\label{e2}
p(t,x)=\frac{p_0(e^{-t}x)}{\int_0^1p_0(e^{-t}x)\,dx}
\end{equation}
and the formula \eqref{e2} defines a dynamical system $\{P^t\}_{t\ge 0}$ on $D_0$ 
given by $P^tp_0(x)=p(t,x)$.

In~\cite{Rudnicki-chaos-den} it is proved that there exists a probability measure $\nu$ 
on $\mathcal B(D_0)$ such that $\nu$  is positive on nonempty open subsets of $D_0$
and the dynamical system
$(D_0,\mathcal B(D_0),\nu,P^t)$ is exact.
The proof of this results runs as follows. First we extend  
the dynamical system $\{S^t\}_{t\ge 0}$ from $X_+$ onto 
the space $L^1_+=\{v\in L^1[0,1]\colon v\ge 0\}$ and define 
the measure $\tilde \mu$ by
$\tilde \mu(A)=\mu(A\cap X_+)$ for $A\in \mathcal B(L^1_+)$, where $\mu$ 
is the previously defined invariant measure with respect to $\{S^t\}_{t\ge 0}$.
Let $Y_0$ be the set which consists of all 
functions $f\in  L^1_+$ such that 
$\int_0^{\varepsilon}f(x)\,dx>0$ for each $\varepsilon>0$.
Then we prove that  $\tilde \mu(Y_0)=1$, the measure  $\tilde \mu$ 
is positive on nonempty open subsets of $D_0$ and the system 
$(Y_0,\mathcal B(Y_0),\tilde\nu,S^t)$ is exact. 
Next we check that the function  $H\colon X_0\to D_0$ defined by 
$Hv=\frac {v}{\|v\|}$ is continuous and satisfies $P^t\circ H=H\circ S^t$ 
for $t\ge 0$. Let $\nu(A)=\tilde\mu(H^{-1}(A))$ for $A\in\mathcal B(D_0)$.
Then the dynamical system $(D_0,\mathcal B(D_0),\nu,P^t)$ has all 
 the required properties.

The method for studying chaos  presented above 
can be extended to a very broad class of partial differential equations
of first order  with one dimensional space variable $x$
and for some hyperbolic equations.
In~\cite{Rudnicki-pd} we study the dynamical system obtained by solving the equation
\begin{equation}
\label{e22}
\frac{\partial u}{\partial t}+
\frac{\partial }{\partial x} (gxu)=
-(m+d)u(t,x)+4du(t,2x)
\end{equation}
which models the size structured cellular population.
We  show that the considered dynamical system 
with an appropriately chosen invariant measure $m$ on the space 
of initial functions $X$ is mixing and the measure $m$ is positive on nonempty open subsets of $X$.

A major challenge is to study the ergodic properties of dynamical systems generated by partial differential equations 
with a multidimensional spatial coordinate $x$. We now consider 
a model of this type given by
the following equation
\begin{equation}
\label{pde-eq-div}
\frac{\partial u}{\partial t}(t,x)+\di(a(x)u(t,x))=\lambda (1-u/K(x))u.
\end{equation}
This equation describes the logistic growth
of a structured population. Any individual is characterized by a vector $x\in {\mathbb R\!}^d$, for example $x$ can be a location in the space
or other parameters as age, maturity, size, etc. These parameters change according to the equation $x'=a(x)$. The function $g(x,u)=\lambda (1-u/K(x))u$ 
describes the population growth per unit time and $u(t,x)$ is  the population density, that is, $\int_Au(t,x)\,dx$ is the number of individuals with parameters in the set $A$ at time $t$. The positive  function $K$ is the local carrying capacity and $\lambda>0$ is the maximum growth rate. 

If $f(x,u)=g(x,u)-u\di a(x)$, then Eq.~\eqref{pde-eq-div} 
can be written in the form
\begin{equation}
\label{pde-eq-1}
\frac{\partial u}{\partial t}(t,x)+a_1(x)\frac{\partial u}{\partial x_1}(t,x)+\dots+a_d(x)\frac{\partial u}{\partial x_d}(t,x)=f(x,u(t,x)),
\end{equation}
$t\ge 0$, $x\in D$, where $D$ is a bounded set in 
${\mathbb R\!}^d$
diffeomorphic with the ball 
$B =\{x\in {\mathbb R\!}^d\colon |x|\le 1\}$.
Ergodic properties of the dynamical system generated by this equation are studied in~\cite{JDE-2023}.  

We assume $a\colon D\to {\mathbb R\!}^d$ be a $C^1$ function, 
$a(\mathbf 0)=\mathbf 0$.
Consider the initial problem 
\begin{equation}
\label{eq-a1}
x'(t)=-a(x(t)),\quad x(0)=x_0\in D. 
\end{equation} 
Denote by $\pi_tx_0$ the solution of \eqref{eq-a1} and
assume that $\lim_{t\to\infty}\pi_tx_0=\mathbf 0$ for all $x_0\in D$.
We also assume that $f\colon D\times \mathbb R\to \mathbb R$ is a $C^1$ function.
We assume that there exist positive constants $A,B$ such that
$f(x,u)u\le A+Bu^2$ for $(x,u)\in D\times \mathbb R$.
Then for each $C^1$ function $v\colon D\to\mathbb R$ 
there exists a unique solution of \eqref{pde-eq-1}
satisfying the initial condition $u(0,x) = v(x)$. 
We define the map $S^tv(x)=u(t,x)$ and extend it to 
the dynamical system $\{S^t\}_{t\ge 0}$ on the space $C(D)$.

We can restrict the dynamical system $\{S^t\}_{t\ge 0}$ to some subspace of $C(D)$. For our purposes we assume additionally that $f(x,0)=0$ for $x\in D$,
$\dfrac {\partial f}{\partial u}(\mathbf 0,0)>0$, 
and there exists $u_0>0$ such that 
$f(\mathbf 0,u_0)=0$, $\dfrac {\partial f}{\partial u}(\mathbf 0,u_0)<0$,
$f(\mathbf 0,u)>0$ for $u\in (0,u_0)$. Then there exists a unique stationary solution $u^+$ of~\eqref{pde-eq-1} such that $u^+(\mathbf 0)=u_0$. 
Let 
\[
X=\{v\in C(D)\colon v(\mathbf 0)=0,\quad v\ge 0,\quad v(x)<u^+(x) 
\textrm{ \,for $x\in D$}\}.  
\]
Then the dynamical system $\{S^t\}_{t\ge 0}$ can be restricted to the set $X$.
One of the results of the paper~\cite{JDE-2023} is that there exists
a probability measure $\mu$ 
on $\mathcal B(X)$ such that $\mu$  is positive on nonempty open subsets of $X$
and the dynamical system
$(X,\mathcal B(X),\mu,S^t)$ is exact.
In order to apply this result to Eq.~\eqref{pde-eq-div} it is sufficient to assume that $\lambda>\di a(\mathbf 0)$.
 
The main differences in the proofs of these results for one-dimensional and multidimensional $x$ variable are the following.
In the construction of the invariant measure $\mu$ instead of the Wiener process, we use the random field called the L\'evy $d$-parameter Brownian motion. Instead of the left-side shift on the space $C[0,\infty)$ 
we use the dynamical system  $\{T^t\}_{t\ge 0}$ on the space  
$C([0,\infty)\times \partial B)$ defined by 
$T^tw(s,y)=w(s+t,y)$, for $s,t\ge 0$ and $y\in \partial B$,
where $\partial B$ is a unit sphere in ${\mathbb R\!}^d$.

\end{document}